\documentclass[11pt,a4paper]{article}

\usepackage{amsmath, amsfonts, amssymb}

\def\be#1\ee{\begin{equation}#1\end{equation}}

\newtheorem{thm}{Theorem}[section]

\renewcommand{\Im}{\operatorname{Im}}

\def\P{{\mathbb{P}}}
\def\R{\mathbb{R}}
\def\E{\mathbb{E}\,}

\def\Z{{\mathbb Z}}

\newenvironment{proof}[1][] {\noindent {\bf Proof#1:} }{\hspace*{\fill}$\square$\medskip\par}

\def\ed#1{ {\mathbf 1}_{ \{#1  \}}}             

\def\al{{\alpha}}
\def\b{{\beta}}

\def\cov{\textrm{Cov}}
\def\EE{{\mathcal E}}
\def\errna{\textrm{ERR}_{\textrm{NA}}}
\def\errap{\textrm{ERR}_{\textrm{A}}^+}
\def\erra{\textrm{ERR}_{\textrm{A}}}

\def\hg{{\widehat{g}}}
\def\hw{{\widehat{w}}}
\def\hh{{\widehat{h}}}
\def\hG{{\widehat{G}_{\le 0}}}
\def\hgamma{{\widehat{\gamma}}}
\def\KK{{\mathcal K}}
\def\L{{\mathbb L}}

\def\PP{{\mathcal P}}
\def\T{{\mathbb T}}

\def\W{{\mathcal W}}

\def \=L{\ {\buildrel\hbox{\scriptsize d }\over =}\ }

\begin{document}

\title{\bf Adaptive Energy Saving Approximation for
Stationary Processes
\\ }
\author{
   Zakhar Kabluchko
   \footnote{M\"unster University, Orl\'eans-Ring 10, 48149 M\"unster, Germany,
   email \ {\tt  zakhar.kabluchko@uni-muenster.de}}
     \and
   Mikhail Lifshits
   \footnote{St.\ Petersburg State University, 7/9 Universitetskaya nab., St.\ Petersburg, 199034 Russia,
   email {\tt mikhail@lifshits.org} and MAI, Link\"oping University.}
}
\date{\today}

\maketitle

\begin{abstract}
We consider a stationary process
(with either discrete or continuous time) and find
an adaptive approximating stationary process combining approximation
quality and supplementary good properties that can be interpreted as additional smoothness
or small expense of energy.
The  problem is solved in terms of the spectral characteristics of the approximated
process by using classical analytic methods from prediction theory.

\end{abstract}
\vskip 1cm

\bigskip

\vfill

\newpage

\section{Main objects and problem setting}

Consider a random process  $B(t)$ with continuous ($t\in \R$) or discrete ($t\in \Z$)
time. We try to approximate $B$ with another process $X$ that, being close to $B$, would
have, appropriately understood, better sample path properties. For example, one may imagine
that a sample path of $B$ is a model of trajectory for some chaotically moving target
while the sample path of $X$ is a pursuit trajectory built upon observations of $B$.
In the most interesting cases (when the time is continuous), the trajectories of $B$ are non-differentiable, while the trajectories of $X$ are required to be smooth.

In this article we assume that $B$ and $X$ are
wide-sense stationary process. The additional requirements on $X$ are stated in terms
of small average expense of generalized energy, the latter notion being formalized below.

\subsection*{Continuous time, stationary process}

Let  $(B(t))_{t\in \R}$ be a centered, complex-valued,
wide-sense stationary process. The latter condition means that $\E |B(t)|^2 < \infty$
and the covariance function of $B$ depends only on the time difference, namely,
\[
  \cov( B(t_1),B(t_2))= \cov( B(t_1-t_2),B(0)):=K_B(t_1-t_2).
\]
As usual, we assume that $K_B(\cdot)$ is continuous.

We will look for an approximating process $(X(t))_{t\in \R}$ such that the pair
$(B,X)$ would be jointly wide-sense stationary, the processes be close to each other
but $X$ spends a small amount of energy (a notion to be specified soon)
in its approximation efforts. 

We call the instant {\it energy} of $X$ at time $t\in\R$ an expression
\[ 
    \EE[X](t):= \left| \sum_{m=0}^M \ell_m X^{(m)}(t)\right|^2,
\]  
where $X^{(m)}$ stands for the $m$-th derivative of $X$ and  $\ell_m$ are some fixed
complex
coefficients. The most natural type of energy is
the {\it kinetic energy} which is just $\al^2 |X^{(1)}(t)|^2$ with some $\al>0$.

The natural goal for us would be the minimization of the functional
\[
   \lim_{T\to\infty} \frac 1T\ \int_0^T \left[ |X(t)-B(t)|^2 + \EE[X](t)\right] dt
\]
combining approximation and energy properties with averaging in time. If, additionally,
the process $X(t)-B(t)$ and all derivatives $X^{(m)}(t)$ are stationary processes in the
strict sense, in many situations the ergodic theorem applies and the limit above is equal to
$\E |X(0)-B(0)|^2 + \E \EE[X](0)$.
In the wide-sense theory, we simplify our task to solving the problem
\be \label{EE0_XB}
  \E |X(0)-B(0)|^2 + \E \EE[X](0) \searrow \min
\ee
and setting aside ergodicity issues.
From the point of view of control theory, the term $\E \EE[X](0)$ may be considered as a
sort of penalty imposing certain smoothness on $X$.
\medskip

Notice that, once the problems of the form \eqref{EE0_XB} are solved, one can easily separate the two terms in
 \eqref{EE0_XB}, solving  (by Lagrange multipliers method) the somewhat more natural problems:

a) Find a process $X$ with minimal expense of energy and reaching prescribed closeness to $B$,
\[
    \E \EE[X](0) \searrow \min \qquad (\textrm{over } X \textrm{  such that }   \E |X(0)-B(0)|^2\le \delta )
\]
for any given $\delta>0$,

b) Find a process $X$ reaching the best possible closeness to $B$ using given amount of energy,
\[
     \E |X(0)-B(0)|^2 \searrow \min    \qquad (\textrm{over } X \textrm{  such that } \E \EE[X](0) \le \EE )
\]
for any given $\EE>0$.
\medskip

We will consider the problem \eqref{EE0_XB} either in the simpler {\it non-adaptive setting},
i.e.\ without any further restrictions on $X$, or in the {\it adaptive setting} by requiring additionally
\[
   X(t) \in \overline{{\rm span}}\{ B(\tau),\tau \le t \big| \L_2(\Omega,\P)  \},
   \quad t\in \R,
\]
where $\overline{{\rm span}}\{ A | H\}$ denotes the closed linear span of a subset $A\subset H$ in a Hilbert space $H$.
In other words, we only allow approximations based on the current and past values of $B$.
The non-adaptive setting was considered in \cite{IKL,KabLi}. In the present paper,
after briefly recalling the corresponding results,  we concentrate on the much more interesting
and difficult adaptive setting.

Our approach in solving  \eqref{EE0_XB} is based on spectral representations of
stationary processes. We will now briefly recall some facts from this theory; see, e.g.,\ \cite{AshG}.

According to the Bochner--Khinchine theorem, the covariance function $K_B(\cdot)$ admits a spectral representation
\[
  K_B(t) = \int_\R e^{itu}\mu(du),
\]
where $\mu$ is a finite measure called the spectral measure of $B$. Moreover, the process $B$ itself
admits a spectral representation
\[
  B(t) = \int_\R e^{itu} \W(du),
\]
where $\W(du)$ is a complex centered random measure with uncorrelated values  on $\R$
controlled by the spectral measure
$\mu$, i.e.\ $\mu(A)=\E |\W(A)|^2$ for any Borel set $A\subset\R$.

Without loss of generality we may restrict our optimization to the class of
approximating process having the form
\be \label{XW}
  X(t) = \int_\R \hg(u) e^{itu} \W(du),
\ee
where $\hg(\cdot)\in \L_2(\R,\mu)$ is an unknown function. For example, if
$X$ is a moving average process,
\be \label{XgB}
   X(t) = \int_\R g(\tau) B(t+\tau) d\tau
\ee
for some weight $g\in \L_1(\R)$, then we have
\begin{eqnarray*}
   X(t) &=& \int_\R g(\tau) \int_\R e^{i(t+\tau)u} \W(du)  d\tau
   \\
       &=& \int_\R  \hg(u)  e^{itu}\W(du),
\end{eqnarray*}
where
\[ 
    \hg(u) :=   \int_\R g(\tau) e^{i\tau u} d\tau
\]
is the inverse Fourier transform of $g$.

Indeed, it is easy to show that every process $X$ that is jointly wide-sense stationary with $B$
can be represented as a sum of two wide-sense stationary processes,
\[
  X(t)= X_1(t) +X_2(t),\qquad t\in \R,
\]
where $X_1$ is a process of the class \eqref{XW} and $X_2(t)$ is uncorrelated with $B(s)$
for all $s,t\in\R$.
It follows that
\[
    \E |X_1(0)-B(0)|^2 + \E \EE[X_1](0) \le  \E |X(0)-B(0)|^2 + \E \EE[X](0),
\]
and reduction to the class \eqref{XW} is justified.

Next, if a process  $X$ has the form \eqref{XW} and
\[
   \int_\R |\hg(u)|^2 u^{2m}  \mu(du)<\infty,
\]
for some positive integer $m$, then the $m$-th mean square derivative of $X$
exists and admits a representation
\[  
   X^{(m)}(t) =  \int_\R \hg(u) (iu)^m e^{itu} \W(du).
\]  
Hence,
\begin{eqnarray*}
   &&   \E |X(0)-B(0)|^2 + \E \EE[X](0)
   \\
   &=&
   \E \left| \int_\R (\hg(u)-1) \W(du) \right|^2  + \E  \left| \int_\R   \hg(u) \sum_{m=0}^M \ell_m (iu)^m   \W(du) \right|^2
   \\
   &=& \int_\R \left[ |\hg(u)-1|^2 + |\hg(u)|^2 |\ell(iu)|^2 \right] \mu(du)
\end{eqnarray*}
with energy polynomial
\[
  \ell(z):= \sum_{m=0}^M \ell_m z^{m}.
\]
Now our problem \eqref{EE0_XB} can be reformulated analytically as
\be \label{EE0}
    \int_\R \left[ |\hg(u)-1|^2 + |\hg(u)|^2|\ell(iu)|^2 \right] \mu(du) \searrow \min,
\ee

Notice that one can consider this problem with more or less arbitrary  function $\ell(\cdot)$
instead of a polynomial.

The spectral condition  equivalent to adaptive setting is
\[
   \hg(u)e^{itu}  \in \overline{{\rm span}}\{ e^{i \tau u},\tau \le t  \big|   \L_2(\R,\mu) \},
   \quad t\in \R.
\]
This condition clearly holds for all $t\in \R$ iff it holds for $t=0$, i.e.\
\[
     \hg  \in \hG:= \overline{{\rm span}}\{ e^{i \tau u},\tau \le 0 \big|  \L_2(\R,\mu) \}.
\]

\subsection*{Discrete time, stationary sequence}

Let  $(B(t))_{t\in\Z}$ be a centered wide-sense stationary sequence, which means that $\E |B(t)|^2 < \infty$ and  its
covariance function depends only on the time difference:
\[
  \cov( B(t_1),B(t_2))= \cov( B(t_1-t_2),B(0)):=K_B(t_1-t_2).
\]
According to the Herglotz theorem, $K_B(\cdot)$ admits a spectral representation
\[
  K_B(t) = \int_\T e^{itu}\mu(du), \qquad t\in \Z,
\]
where $\mu$ is a finite measure on $\T:=[-\pi,\pi)$ called the spectral measure of $B$.
The sequence $B$ itself admits a spectral representation
\[
  B(t) = \int_\T e^{itu} \W(du), \qquad t\in \Z,
\]
where $\W(du)$ is a complex centered random measure with uncorrelated values
on $\T$ controlled by $\mu$.

As in \eqref{XW}, we search an approximating sequence $(X(t))_{t\in \Z}$, in the form
\[
  X(t) = \int_\T \hg(u) e^{itu} \W(du),
\]
where $\hg(\cdot)\in \L_2(\T,\mu)$. For example, if $X$ is a moving average
sequence,
\[ 
   X(t) = \sum_{\tau\in \Z} g(\tau) B(t+\tau)
\] 
for some summable weight $g$, then we have
\begin{eqnarray*}
   X(t) = \int_\T  \hg(u)  e^{itu}\W(du), \qquad t\in\Z,
\end{eqnarray*}
where
\[ 
    \hg(u) :=   \sum_{\tau\in\Z} g(\tau) e^{i\tau u}
\] 
is the inverse Fourier transform of $g$.

In the discrete case the notion of energy should be modified by replacing the (right)
derivatives with their discrete analogues, e.g. $X(t+1)-X(t)$ for $X'(t)$, $X(t+2)-2X(t+1)+X(t)$ for $X''(t)$,
etc. Therefore, the instant energy of $X$ takes the form
\[ 
    \EE[X](t):= \left|  \sum_{m=0}^{M} \ell_m X(t+m) \right|^2,
\] 
and using the integral representation
\begin{eqnarray*}
    \sum_{m=0}^{M} \ell_m X(t+m)
    &=&\int_\T \hg(u) \left( \sum_{m=0}^{M} \ell_m e^{i m u}\right) e^{i t u} \W(du)
    \\
    &:=& \int_\T \hg(u)\, \ell(e^{i u})\, e^{i t u} \W(du)
\end{eqnarray*}
with the polynomial
\[ 
  \ell(z):= \sum_{m=0}^{M} \ell_m z^{m}
\] 
we have
\[
    \E \EE[X](t)  = \int_\T |\hg(u)|^2 |\ell(e^{iu})|^2 \mu(du),  \qquad t \in\Z.
\]

The discrete-time version of problem \eqref{EE0} becomes
\be \label{EE0_discr}
    \int_\T \left[ |\hg(u)-1|^2 + |\hg(u)|^2 |\ell(e^{iu})|^2 \right] \mu(du) \searrow \min.
\ee
Again, one can also consider this problem with arbitrary
function $\ell(\cdot)$ instead of the polynomial. The discrete-time  analogue of kinetic energy corresponds to the increment $\al(X(t+1)-X(t))$,
i.e.\ to the polynomial  $\ell(z)= \al(z-1)$.

One can consider the problem \eqref{EE0_discr} either in the non-adaptive setting, or in
the adaptive setting by requiring additionally
\[
   \hg \in \hG  :=\overline{{\rm span}}\{ e^{i\tau u},\tau \le 0, \tau\in\Z  \big| \L_2(\T,\mu)  \} .
\]

\section{First step: solution of the non-adaptive problem}
\subsection{Continuous time}
As we have seen for continuous-time setting,  our problem states in \eqref{EE0} as
\be \label{EE0a}
    \int_\R \left[ |\hg(u)-1|^2 +  |\hg(u)|^2 |\ell(iu)|^2 \right] \mu(du) \searrow \min,
\ee
For any complex numbers $\hg$ and $\ell$  we have an identity
\be \label{EE0a_next}
  |\hg-1|^2 +  |\hg|^2 |\ell|^2
  =
   \left( |\ell|^2 +1\right) \left| \hg-\frac{1}{ |\ell|^2 +1}\right|^2
    +\frac{|\ell|^2}{ |\ell|^2 +1}\ .
\ee
Therefore, in the non-adaptive setting, where no further
restrictions are imposed on the function $\hg$, the solution
to \eqref{EE0a} is given by the function
\be \label{hg_star}
   \hg_*(u):= \frac{1}{ |\ell(iu)|^2 +1}\ , \qquad u\in \R,
\ee
depending on the energy form $\ell$ but not on the spectral measure $\mu$. The minimum in \eqref{EE0a}
is equal to
\[
   \errna :=    \int_\R  \frac{|\ell(iu)|^2}{ |\ell(iu)|^2 +1}\ \mu(du).
\]
It is natural to call this quantity {\it non-adaptive approximation error}.
In the control theory  the term {\it problem cost} is also used.

In the simplest case of the kinetic energy $\ell(z)=\al z$,
 where $\al>0$ is a scaling parameter, we get
\be \label{errna_kin}
   \errna :=    \int_\R  \frac{\al^2 u^2}{ \al^2 u^2 +1}\ \mu(du).
\ee
Since  $\hg_*(u)= \frac{1}{ \al^2 u^2 +1}$ is the inverse Fourier transform for
\[
    g_*(t) := \frac 1{2\al}\ \exp\{-|t|/\al \}, \qquad t\in \R,
\]
we conclude that the solution to non-adaptive problem with kinetic energy
for stationary processes
is given, as suggested in \eqref{XgB}, by the moving average process
\be \label{XgB_star}
   X(t) = \frac 1{2\al} \int_\R \exp\{-|\tau|/\al\} B(t+\tau) d\tau.
\ee
Notice that this solution is indeed non-adaptive because the future values of $B$
are involved into approximation.
The formula \eqref{hg_star} was obtained in \cite{IKL}, see also \cite{KabLi} for the case of
kinetic energy.

However, if adaptivity restriction is imposed on $\hg$, then
\eqref{hg_star} does not apply and we have to minimize the spectrum-dependent integral.
By \eqref{EE0a_next}, the problem \eqref{EE0a} reduces to
\be \label{EE_add}
    \int_\R   \left| \hg(u)-\frac{1}{ |\ell(iu)|^2 +1}\right|^2   \left( |\ell(iu)|^2 +1\right)
    \mu(du) \searrow \min.
\ee
This minimum (taken over $\hg \in \hG$) will be called {\it additional adaptivity error} and denoted
$\errap$. This is the price we must pay for not knowing the future. The {\it total
approximation error}, i.e.\ the minimum in \eqref{EE0a} over $\hg \in \hG$ is then equal to
\[
   \erra := \errna + \errap.
\]
These formulae were obtained in \cite{IKL}.

\subsection{Discrete time}
In the discrete-time setting, the situation is completely similar because the problem
\eqref{EE0_discr} differs from \eqref{EE0a} only by replacing the spectral domain $\R$
with $\T$ and $\ell(iu)$ with $\ell(e^{iu})$. Therefore, we obtain
the expression for the non-adaptive  error
\[
   \errna :=    \int_\T  \frac{|\ell(e^{iu})|^2}{ |\ell(e^{iu})|^2 +1}\ \mu(du)
\]
attained by the minimizer
\[
   \hg_*(u):= \frac{1}{ |\ell(e^{iu})|^2 +1}\ , \qquad u\in \T.
\]
In the simplest case of the  discrete-time kinetic energy $\ell(z)=\al (z-1)$,
we have
\be \label{series1_adapt}
   \hg_*(u)
   = \frac{1}{\sqrt{1+4\al^2}} \left( 1+ \sum_{k=1}^\infty \b^{-k}
   \left( e^{i k u}+ e^{-i k u}\right)\right),
\ee
where
\be \label{beta}
  \b=\frac{2\al^2+1+\sqrt{1+4\al^2}}{2\al^2}>1.
\ee
The analogue of  \eqref{XgB_star} is
\[ 
  X(t) =  \frac{1}{\sqrt{1+4\al^2}} \left( B(t) + \sum_{k=1}^\infty \b^{-k}
   \left( B(t+k)+ B(t-k) \right)\right),
\] 
while
\be \label{errna_kin_discr}
   \errna =    \int_\T  \frac{\al^2 |e^{iu}-1|^2}{ \al^2 |e^{iu}-1|^2 +1}\ \mu(du).
\ee


\section{Solutions to adaptive approximation problem}

Recall that adaptive approximation problem for continuous-time processes was reduced
in \eqref{EE_add} to solving the problem
\be \label{EE_add_b}
    \int_\R   \left| \hg(u)-\frac{1}{ |\ell(iu)|^2 +1}\right|^2   \left( |\ell(iu)|^2 +1\right)
    \mu(du) \searrow \min
\ee
over $\hg\in \hG$. This looks very much as a classical prediction problem, except for
the function to be approximated: in our setting it is $\frac{1}{ |\ell(iu)|^2 +1}$, while in prediction
problem it is $e^{i\tau u}$ for some $\tau>0$. Therefore we may either directly reduce
the approximation problem to the prediction problem, or to use methods that are usually
used for solving the prediction problems. The latter way seems to be more efficient and general.

\subsection{Straight reduction to prediction problems}
\subsubsection{Continuous time}
Consider continuous-time setting.
Assume that we have an appropriate (to be made precise a bit later) factorization
\be \label{fact_c}
   |\ell(iu)|^2 +1 = \lambda_\ell(u)\,  \overline{\lambda_\ell(u)}= |\lambda_\ell(u)|^2,
   \qquad u \in \R.
\ee
Then the left-hand side of \eqref{EE_add_b} becomes
\be \label{EE_add_c}
       \int_\R   \left| \lambda_\ell(u)\, \hg(u) -  \frac{1}{\overline{\lambda_\ell(u)}} \right|^2  \mu(du).
\ee
Recall that the classical prediction problem is
\be \label{EE_classical_prediction}
    \int_\R   \left| \hg(u)- e^{i\tau u} \right|^2   \mu(du) \searrow \min,
    \qquad \tau\in\R,
\ee
over $\hg\in\hG$. Let $\widehat{q}_*^{(\tau,\mu)}$  denote the solution
of this problem.
The solution of the general prediction problem
\[
    \int_\R   \left| \hg(u)- v(u) \right|^2   \mu(du) \searrow \min
\]
is linear in $v$. Therefore, if there is a representation
\be \label{lambar}
  \frac{1}{\lambda_\ell(u)} = \int_0^\infty e^{-i\tau u} \nu_\ell(d\tau),
   \qquad u\in \R,
\ee
with some finite complex measure $\nu_\ell$ depending on the energy polynomial $\ell(\cdot)$, then the function
\[
      \widehat{q}_*^{(\ell,\mu)}(u) := \int_0^\infty  \widehat{q}_*^{(\tau,\mu)}(u) \overline{\nu_\ell}(d\tau),
     \qquad u \in \R,
\]
belongs to $\hG$ and satisfies
\[
       \int_\R   \left|   \widehat{q}_*^{(\ell,\mu)}(u)  -  \frac{1}{\overline{\lambda_\ell(u)}} \right|^2  \mu(du)
			 \le
       \int_\R   \left| \hg(u)  -  \frac{1}{\overline{\lambda_\ell(u)}} \right|^2  \mu(du), \qquad \hg\in \hG.
\]
It also follows from representation \eqref{lambar} that for any $\hg\in\hG$ we have $ \frac{1}{\lambda_\ell}\, \hg\in\hG$.

Now we impose another assumption on the factorization \eqref{fact_c}:
\be \label{lam_hg}
  \textrm{If } \hg\in\hG \textrm{ and } \lambda_\ell\ \hg  \in \L_2(\R,\mu), \textrm{ then one must have }
  \lambda_\ell\ \hg\in\hG.
\ee
If conditions  \eqref{lambar} and  \eqref{lam_hg} are verified, then the function
\[
     \hg_*(u):=\lambda_\ell(u)^{-1} \widehat{q}_*^{(\ell,\mu)}(u)
\]		
minimizes \eqref{EE_add_c} over $\hg\in \hG$ and thus solves  the problem  \eqref{EE_add_b}.
Indeed, let  $\hg\in\hG$. If $\lambda_\ell \hg  \not\in \L_2(\R,\mu)$, then expression
\eqref{EE_add_c} is infinite because the bounded function $1/\overline{\lambda_\ell}$
belongs to $\L_2(\R,\mu)$. Let now $\lambda_\ell \hg \in \L_2(\R,\mu)$.
Then, using the optimality of $\widehat{q}_*^{(\ell,\mu)}(u)$ and   inclusion $ \lambda_\ell\,\hg\in\hG$
we have
\begin{eqnarray*}
 \int_\R   \left| \lambda_\ell(u)\, \hg_*(u) -  \frac{1}{\overline{\lambda_\ell(u)}} \right|^2  \mu(du)
&=&
 \int_\R   \left|  \widehat{q}_*^{(\ell,\mu)}(u) -  \frac{1}{\overline{\lambda_\ell(u)}} \right|^2  \mu(du)
\\
&\le&
 \int_\R   \left| \lambda_\ell(u)\, \hg (u) -  \frac{1}{\overline{\lambda_\ell(u)}} \right|^2  \mu(du)
\end{eqnarray*}
and the problem is solved.

We stress that for polynomials a representation with properties \eqref{lambar} and  \eqref{lam_hg}
is always possible.
Indeed, let $\ell(\cdot)$ be a polynomial of degree $M$ with complex coefficients.
Then for all real $u$ we have
\[
   1+|\ell(iu)|^2 = 1+ \ell(iu) \overline{\ell(iu)} := \PP(u),
\]
where $\PP$ is a polynomial of degree $2M$ with real coefficients. Therefore, if
$\b$ is a root of $\PP$, then $\overline{\b}$ also is its root. Notice also that $\PP$ has no real roots.
Thus we may write
\[
   \PP(u) = C \prod_{m=1}^M (u-\b_m)  (u-\overline{\b_m}),
\]
where $\Im(\b_m)>0$ and $C>0$ (which follows by letting $u=0$). Finally, we obtain
\[
   1+|\ell(iu)|^2 = \PP(u) = \lambda_\ell(u) \ \overline{ \lambda_\ell(u)}, \qquad u\in \R,
\]
where
\be \label{fac_lam3}
    \lambda_\ell(u):= C^{1/2}  \prod_{m=1}^M (u-\b_m).
\ee
In this case the representation
\[
   \frac{1}{u-\b_m} =
  i \int_0^\infty \exp(-i(u-\b_m)\tau) d\tau
\]
holds true and the existence of representation \eqref{lambar} for
$\frac{1}{\lambda_\ell}$ follows.

In order to verify \eqref{lam_hg}, notice that, since the polynomial
$\lambda_\ell$ is bounded away from zero, condition
$\lambda_\ell\ \hg  \in \L_2(\R,\mu)$ is equivalent to
\[
  \int_\R |u|^{2m}\, |\hg(u)|^2 \mu(du)<\infty,
  \qquad 1\le m \le M.
\]
Using this fact with $m=1$ we see that the functions
\[
   \hg_\delta(u):= \frac{1-e^{-i\delta u}}{i\delta}\, \hg(u) \in \hG
\]
converge to the function $u \hg(u)$ in  $\L_2(\R,\mu)$ as $\delta\to 0$. Therefore,
the limit also belongs to  $\hG$. One continues by induction and concludes
that all functions  $u^m \hg(u)$ for $1\le m\le M$ belong to  $\hG$.
Obviously, the same is true for their linear combination: $\lambda_\ell\ \hg \in \hG$.

For example, for continuous-time kinetic energy $\ell(z)=\al z$, we may use the factorization
\[ 
 |\ell(iu)|^2+1 = \al^2 u^2 +1
 := \lambda_\ell(u) \overline{\lambda_\ell(u)},
\] 
where $\lambda_\ell(u):= 1+i\al u$
and
\be \label{lambinv_kc}
    \frac 1{\lambda_\ell(u)}
    =\frac{1}{\al }  \int_{0}^\infty  e^{-\tau/\al} e^{-i \tau u} d\tau.
\ee
\medskip

\subsubsection{Discrete time}
In the discrete-time setting, one should only replace $\R$ with $\T$ and
$\ell(iu)$ with $\ell(e^{iu})$ in \eqref{EE_add_c}. Now we need a factorization
\[
   |\ell(e^{iu})|^2 +1 = \lambda_\ell(u) \overline{\lambda_\ell(u)}= |\lambda_\ell(u)|^2,
   \qquad u \in \T,
\]
with $\lambda_\ell(\cdot)^{-1}$ admitting a representation
\be \label{lambar_d}
   \frac{1}{\lambda_\ell(u)} = \sum_{\tau=0}^\infty  \nu_\ell(\tau) e^{-i\tau u} ,
   \qquad u\in \T,
\ee
in place of \eqref{lambar}, and satisfying an obvious analogue of \eqref{lam_hg}, i.e.\
\be \label{lam_hg_d}
  \textrm{If } \hg\in\hG \textrm{ and } \lambda_\ell\ \hg \in \L_2(\T,\mu),
	\textrm{ then one must have } \lambda_\ell\ \hg\in\hG.
\ee
Then the function
\[
   \hg_*(u)
   := \lambda_\ell(u)^{-1} \sum_{\tau=0}^\infty \nu_\ell(\tau) \widehat{q}_*^{(\tau,\mu)}(u),
     \qquad u \in \T,
\]
where the function $\widehat{q}_*^{(\tau,\mu)}$ is the minimizer in the classical prediction problem, cf.\ \eqref{EE_classical_prediction}, provides a solution for our problem.

We explain now how to construct the required factorizations for the typical forms of energy
represented by arbitrary complex polynomials $\ell(\cdot)$. Indeed, let
\[
    \ell(z):=\sum_{m=0}^M \ell_m z^m
\]
with $\ell_M\not=0$. Then for $z$ on the unit circle
\begin{eqnarray*}
   1+|\ell(z)|^2 &=& 1+ \ell(z) \overline{\ell(z)}
\\
     &=& 1 + \left(\sum_{m=0}^M \ell_m z^m\right)  \left(\sum_{m=0}^M \overline{\ell_m} z^{-m}\right)
\\
     &:=& \frac{\PP(z)}{z^M},
\end{eqnarray*}
where
\[
   \PP(z):= \sum_{m=0}^{2M} p_m z^m
\]
is a polynomial of degree at most $2M$ with coefficients satisfying Hermitian symmetry condition
$p_{2M-m}=\overline{p_m}$. Due to this symmetry, if
$\b\not=0$ is a root of $\PP$, then $1/\overline{\b}$ also is its root. Notice also that $\PP$ has no
roots on the unit circle.
Assume temporarily that $\ell_0\not=0$. Then $p_0=\ell_0\overline{\ell_M}\not=0$,
hence zero is not a root of $\PP$
and we may write
\begin{eqnarray*}
    1+|\ell(z)|^2 &=&  \frac{\PP(z)}{z^M} = C \prod_{m=1}^M (z-\b_m)  (z-\frac1{\overline{\b}_m}) \frac1z
\\
   &=&  \frac{ (-1)^M C} { \prod_{m=1}^M \overline{\b}_m} \prod_{m=1}^M (z-\b_m)(\overline{z}-\overline \b_m)
\end{eqnarray*}
for some  complex $C$ and $|\b_m|>1$. Letting, say, $z=1$, shows that the exterior constant is positive:
\[
   R:= \frac{ (-1)^M C} { \prod_{m=1}^M \overline{\b}_m} >0.
\]
Hence, we have the factorization
\[
    1+|\ell(e^{iu})|^2  = \lambda_\ell(u) \ \overline{ \lambda_\ell(u)}, \qquad u\in \T,
\]
with
\be \label{fac_lam2}
    \lambda_\ell(u):= R^{1/2}
    \prod_{m=1}^M (e^{-iu}-\overline \b_m).
\ee
 The proof of the required properties is the same as in the
continuous-time case. It is therefore omitted.

Finally, notice that the temporary assumption $\ell_0\not=0$ may be easily dropped.
Indeed, in the general case we may always write
$\ell(z)=z^k \widetilde\ell(z)$ with some $k\le M$ and $\widetilde\ell(0)\not=0$.
Then the factorization for $\widetilde\ell$ also applies to $\ell$ because
$1+|\ell(\cdot)|^2=1+|\widetilde\ell(\cdot)|^2$ on the unit circle.

For discrete-time  kinetic energy $\ell(z)=\al(z-1)$, we may use a factorization
\be \label{fact_kd2}
 |\ell(e^{iu})|^2+1 = \frac{\al^2}{\b} (e^{-iu}-\b)(e^{iu}-\b)
 := \lambda_\ell(u)\, \overline{\lambda_\ell(u)}
\ee
with
\[
   \lambda_\ell(u):= \frac{\al}{\sqrt{\b}} (e^{-iu}-\b)
\]
and $\b$ from \eqref{beta}. In this case we see that
\be \label{lambar_kd}
    \frac{1}{ \lambda_\ell(u)}
    =\frac{-1}{\al\sqrt{\b}} \cdot \frac{1}{1-e^{-iu}/\b }
    =  \frac{-1}{\al\sqrt{\b}} \sum_{\tau=0}^\infty  \b^{-\tau} e^{-i \tau u}
\ee
holds as a version of \eqref{lambar_d}.

\subsection{Application of prediction technique}
\subsubsection{Discrete time}

We first recall few notions used in the analytical prediction technique.
Let
\begin{align*}
L^2_{\le 0}&:= \overline{\rm span}\{e^{i \tau u}, \tau\le 0, \tau \in \Z \big| \L_2(\T,\Lambda)\}, \\
L^2_{> 0}&:= \overline{\rm span}\{e^{i \tau u}, \tau>0, \tau \in \Z \big| \L_2(\T,\Lambda)\},
\end{align*}
where $\Lambda$ denotes Lebesgue measure, be the spaces of spectrally negative
and spectrally positive functions.
We will need a special class of {\it outer} functions. We do not recall the direct
formal definition of an outer function, cf.\ \cite[p.342]{Ru}; instead, we use the following
characterization, cf.\  \cite[Theorem 17.23]{Ru}: a function $\gamma \in \L_2(\T,\Lambda)$
is a conjugated outer function iff
\be \label{outer_d}
     \overline{\rm span}\{ \gamma e^{i\tau u}, \tau\le 0, \tau \in\Z \big|  \L_2(\T,\Lambda) \} =  L^2_{\le 0}.
\ee
We stress that these functions are {\it complex conjugated} to outer functions as defined in Rudin
\cite{Ru}. In the sequel, however, we call them simply ``outer functions''; this omission should not lead
to any misunderstanding.

Now we pass to the optimization problem. By \eqref{EE_add}, we have to compute
\[
   \errap =  \min_{\hg\in\hG}  \int_\T   \left| \hg(u)-\frac{1}{ |\ell(e^{iu})|^2 +1}\right|^2
   \left( |\ell(e^{iu})|^2 +1  \right)  \mu(du),
\]
where
\[
  \hG:= \overline{\rm span}\{ e^{i\tau u}, \tau\le 0, \tau \in\Z \big|  \L_2(\T,\mu) \}.
\]
Assume that the spectral measure has a density on $\T$ satisfying Kolmogorov's regularity condition,
i.e.\ $\mu(du)=f(u)du$ and
\be \label{Kolm}
 \int_\T |\ln f(u)| \ du < \infty.
\ee
The classical prediction technique suggests to find factorizations
\begin{eqnarray}
  \label{fact_gam}
  f(u)  =  \gamma_f(u) \overline{\gamma_f(u)}= |\gamma_f(u)|^2,
  \qquad u \in \T,
\\  \label{fact_lam}
	 |\ell(e^{iu})|^2 +1 = \lambda_\ell(u) \overline{\lambda_\ell(u)}= |\lambda_\ell(u)|^2,
   \qquad u \in \T,
\end{eqnarray}
with $\gamma_f$ being an outer function and $\lambda_\ell$, as above,
satisfying conditions \eqref{lambar_d} and \eqref{lam_hg_d}.
Notice that assumption \eqref{Kolm} implies the existence of factorization \eqref{fact_gam},
cf.\ \cite[Theorem 17.16]{Ru}. Factorization \eqref{fact_lam}
in the case of polynomial $\ell$ was given in \eqref{fac_lam2}.

\begin{thm}
Let $Q_{>0}$ be the orthogonal projection of $\gamma_f / \overline{\lambda_\ell}$ onto $L^2_{>0}$ in the Hilbert space $L^2 (\T, \Lambda)$.
Then the optimal adaptive approximation is given by $X(t) = \int_\T \hg_*(u) e^{itu} \W(du)$ with
$$
\hg_* (u) = \frac 1 {|\lambda_\ell|^2} - \frac{Q_{>0}}{\lambda_\ell \gamma_f}.
$$
The error of the adaptive approximation is given by
$
\errap = \|Q_{>0}\|_2^2
$.
\end{thm}
\begin{proof}
We have
\begin{align}
  \errap
  &=
  \min_{\hg \in \hG}  \int_\T  \left| \lambda_\ell(u)\, \hg(u) -  \frac{1}{\overline{\lambda_\ell(u)}} \right|^2  \mu(du) \label{errap_dL}\\
  &=
   \min_{\hg \in \hG}  \int_\T
  \left| ( \lambda_\ell \gamma_f \hg)(u)-\frac{\gamma_f(u)}{ \overline{\lambda_\ell(u) }} \right|^2 \ du.  \notag
\end{align}
Consider arbitrary $\hg\in\hG$. Without loss of generality we may assume that
$\lambda_\ell\, \hg\in \L_2(\T,\mu)$ (otherwise  the integral in \eqref{errap_dL}  is infinite).
Then by \eqref{lam_hg_d} we have $\lambda_\ell \, \hg\in \hG$, which is equivalent to
\[
  \lambda_\ell \gamma_f \hg
  \in \overline{\rm span}\{ \gamma_f e^{i\tau u}, \tau\le 0, \tau\in\Z \big|   \L_2(\T,\Lambda) \}
  = L^2_{\le 0},
\]
where the latter equality holds by \eqref{outer_d} because $\gamma_f$ is an outer function.

On the other hand, since  $\gamma_f\in L^2(\T,\Lambda)$ and
$|\lambda_\ell|\ge 1$, we have
\[
   Q:=\frac{\gamma_f}{ \overline{\lambda_\ell }} \in \L_2(\T,\Lambda).
\]
Consider the unique orthogonal decomposition in $\L_2(\T,\Lambda)$
\[
   Q:=\frac{\gamma_f}{ \overline{\lambda_\ell }} = Q_{\le 0}+  Q_{>0}
\]
with $Q_{\le 0}\in L^2_{\le 0}$ and $Q_{> 0}\in L^2_{>0}$.
Due to the orthogonality of the spaces $L^2_{\le 0}$ and $L^2_{>0}$,
we clearly have
\[
   \errap \ge \|Q_{>0}\|_2^2.
\]
Moreover,  the equality
\be \label{errQ}
   \errap = \|Q_{>0}\|_2^2
\ee
is attained whenever
\be \label{gstar}
  \hg=\hg_* := \frac{Q_{\le 0}}{ \lambda_\ell \gamma_f}
  = \frac{Q- Q_{>0}}{ \lambda_\ell \gamma_f}
  = \frac 1{|\lambda_\ell|^{2}} -\frac{Q_{>0}}{ \lambda_\ell \gamma_f}.
\ee
It remains to prove that $\hg_*\in \hG$.
Since $\gamma_f$ is an outer function, we have by \eqref{outer_d}
\[
  Q_{\le 0} \in L^2_{\le 0}
  =\overline{\rm span}\{ \gamma_f e^{i\tau u}, \tau\le 0,\tau\in\Z \big| \L_2(\T,\Lambda) \},
\]
which is equivalent to
\[
  \frac{Q_{\le 0}} {\gamma_f}
  \in \overline{\rm span}\{  e^{i\tau u}, \tau\le 0, \tau\in\Z \big| \L_2(\T,\mu) \} = \hG.
\]
Finally, we obtain from \eqref{lambar_d} the required inclusion
\[
  \hg_* =  \frac{Q_{\le 0}} {\lambda_\ell \gamma_f} \in \hG,
\]
which completes the proof.
\end{proof}

For the discrete-time kinetic energy $\ell(z)= \al(z-1)$ we may proceed further as follows by using decomposition
\eqref{fact_kd2}. Since $\gamma_f$ is an outer function, it belongs to $L^2_{\le 0}$. By taking the Fourier series expansion
\be \label{gammaj}
  \gamma_f(u) := \sum_{j=0}^\infty \hgamma_j e^{-i j u},
  \qquad u\in \T,
\ee
and multiplying \eqref{lambar_kd} and \eqref{gammaj} we obtain
\begin{eqnarray*}
 Q(u)&=& \frac{-1}{\al\sqrt{\b}}  \sum_{\tau=0}^\infty  \b^{-\tau} e^{i\tau u} \cdot
     \sum_{j=0}^\infty \hgamma_j e^{-i j u}
\\
   &=& \frac{-1}{\al\sqrt{\b}}  \sum_{n=-\infty}^\infty \left[
     \left( \sum_{j=\max(-n,0)}^\infty \hgamma_j   \b^{-(n+j)} \right) e^{i n u} \right]
\\
   &=& \frac{-1}{\al\sqrt{\b}}  \sum_{n=-\infty}^\infty \left[
     \left( \sum_{j=\max(-n,0)}^\infty \hgamma_j  \b^{-j} \right) \b^{-n} e^{i n u} \right].
\end{eqnarray*}
Hence,
\begin{eqnarray} \nonumber
  Q_{>0} (u) &=&
  \frac{-1}{\al\sqrt{\b}}     \left( \sum_{j=0}^\infty \hgamma_j  \b^{-j} \right)
   \left( \sum_{n=1}^\infty  \b^{-n} e^{i n u} \right)
\\ \label{Q_kd1}
  &=& \frac{- \KK}{\al\sqrt{\b}}   \left( \sum_{n=1}^\infty
     \b^{-n} e^{i n u} \right)
\\ \label{Q_kd2}
  &=&
     \frac{- \KK }{\al\sqrt{\b}} \  \frac{e^{iu}}{\b -e^{iu}}
   =
     \frac{\KK }{\b} \  \frac{e^{iu}}{ \overline{\lambda_\ell(u) }}
\end{eqnarray}
with
\be \label{Kf}
    \KK := \sum_{j=0}^\infty \hgamma_j  \b^{-j}.
\ee
By  \eqref{errQ} and \eqref{Q_kd1}, it follows that
\begin{eqnarray} \nonumber
  \errap &=& \|Q_{>0}\|_2^2 = \frac{1}{\al^2\b} |\KK|^2
  \frac{2\pi}{\b^2-1}
  \\ \label{err_kd}
  &=& \frac{2\pi}{\al^2\b^2} |\KK|^2  \frac{1}{\b-1/\b}
  = \frac{2\pi}{\b^2\sqrt{1+4\al^2}} \, |\KK|^2,
\end{eqnarray}
where we used the identity
\be \label{id_bmb}
   \b-\frac 1\b=\frac{\sqrt{1+4\al^2}}{\al^2}.
\ee
We also have from  \eqref{gstar} and  \eqref{Q_kd2}
\be \label{gstar_kd}
  \hg_*(u)
  =
  \frac 1 {|\lambda_\ell(u)|^{2}} \left( 1 -\frac{\KK e^{iu}} {\b \gamma_f(u)}\right)
  =
  \frac {\left( 1 -\frac{\KK e^{iu}} {\b \gamma_f(u)}\right)} {2\al^2 (1-\cos u) + 1},
  \quad u\in\T.
\ee

In the case when $\ell$ is an arbitrary polynomial, we can use 
\eqref{fac_lam2} to construct a partial fraction decomposition of 
$1/\overline{\lambda_l}$ into a linear combination of fractions of the form 
$1/ (1-e^{iu}\beta_m^{-1})$ provided the numbers $\beta_m$ are pairwise 
distinct. Then we can apply the above considerations to every fraction 
separately.

It is possible to provide an explicit formula for the outer function $\gamma_f$ 
and the constant $\KK$ in terms of the spectral density $f$. To this end, 
consider the function
$$
   q(z) :=\exp\left\{\frac 1 {2\pi} \int_\T \frac{e^{iu} + z}{e^{iu} - z} 
           \ln \sqrt{f(u)} \, du\right\}, \quad |z|<1.
$$
By \cite[Theorem 17.16]{Ru}, the radial limits of its absolute value $|q|$ are 
Lebesgue-a.e.\ given by
$$
   \lim_{r\uparrow 1} |q(r e^{iu})| = \sqrt{f(u)}, \quad u \in\T.
$$
Since the function $\sqrt{f}$ is square integrable,  \cite[Theorem 17.16(c)]{Ru} 
implies that the function $q$ belongs to the Hardy space $H^2$ on the unit disc. 
Defining
$$
   \gamma_f(u) = \lim_{r\uparrow 1} \overline{q(r e^{iu})}, \quad u\in\T,
$$
we clearly have $|\gamma_f(u)|^2 = f(u)$ for $u\in\T$. Also, $\gamma_f$ is a 
(complex conjugate of an) outer function by \cite[Definition 17.14]{Ru}. 
The Fourier series representation of $\gamma_f$ given in \eqref{gammaj} 
translates into a Taylor series representation of $q$ as follows:
$$
   q(z) = \sum_{j=0}^\infty \overline{\hgamma_j} z^j, \quad |z|<1.
$$
Returning to the case $\ell(z)= \al(z-1)$, it follows from \eqref{Kf} that
$$
    \KK = \overline{q(1/\beta)} = \exp\left\{\frac 1 {4\pi} 
    \int_\T \frac{e^{-iu} + \beta^{-1}}{e^{-iu} - \beta^{-1}} \ln f(u) \, du\right\}.
$$
Recalling \eqref{err_kd} and doing straightforward transformations, 
we arrive at the following

\begin{thm}
In the discrete-time case with $\ell(z) = \alpha (z-1)$, the additional 
adaptivity error is given by
\be\label{ERRAP_discr}
   \errap = \frac{2\pi}{\beta^2 \sqrt{1+4\alpha^2}} \exp\left\{\frac 1 {2\pi} 
   \int_{\T} \frac{\beta^2-1}{\beta^2+1-2\beta \cos u} \ln f(u)\, du \right\},
\ee
where $\beta = (2\al^2+1+\sqrt{1+4\al^2})/(2\al^2)$.
\end{thm}

In the above argument, we assumed that $\int_{\T} |\ln f(u)|\ du < \infty$, but 
\eqref{ERRAP_discr} remains valid even when $\int_{\T} |\ln f(u)|\ du =\infty$. 
Indeed, since $f$ is a density, the latter condition is equivalent to
\be\label{Kolm_anti}
    \int_\T \ln f(u)\, du = -\infty,
\ee
and \eqref {ERRAP_discr} states that $\errap = 0$. This result is easy to explain: 
It is known (e.g.,\ \cite[pp.\ 48--50]{Hoff}) that under~\eqref{Kolm_anti} (or if 
$f$ does not exist at all, see \cite[Corollary 1 on p.\ 46]{Hoff}), it is possible 
to predict the future of the process $X$ on the basis of its past \textit{perfectly}, 
so that there is no difference between the non-adaptive and the adaptive approximation.

Let us look at our approximation problem when $\alpha\downarrow 0$ which means 
that we give less importance to the kinetic energy of the approximating process 
compared to the closeness of the processes. As $\alpha\downarrow 0$ we have 
$\beta = \alpha^{-2} + O(1)\to\infty$, and  \eqref{ERRAP_discr} yields
\be\label{errap_as_al_to_0}
     \errap \sim   
     2\pi \al^4 \exp\left\{\frac 1 {2\pi} \int_{\T} \ln f(u)\, du \right\}.
\ee
The right-hand side looks very much like the classical Kolmogorov formula. Let us 
explain this similarity.  Recall that the classical prediction problem asks to 
predict $B(\tau)$ on the basis of $B(0), B(-1),\ldots$, for $\tau=1,2,\ldots$. 
In particular, in the case of one-step prediction $\tau=1$, the Kolmogorov formula, 
see \cite[Theorem 5.8.1]{BD} or \cite[pp.\ 48--50]{Hoff}, states that the mean 
square error of the optimal prediction is given by
$$
    \sigma^2_{\text{pred}} := \min_{\hh \in \hG} \int_{\T} |e^{iu} - \hh(u)|^2 du  
    = 2\pi \exp\left\{\frac 1 {2\pi} \int_{\T}\ln f(u)\, du \right\}.
$$
 By \eqref{series1_adapt}, the optimal non-adaptive strategy is given by
$$
\hg_{*}^{\text{(nad)}} (u) = 1 + (e^{iu} + e^{-iu} -2) \alpha^2 + O(\alpha^4).
$$
For the optimal adaptive strategy, it is therefore natural to make the ansatz
$$
\hg_{*}^{\text{(ad)}} (u) = 1 + (\hw(u) + e^{-iu} -2) \alpha^2 + O(\alpha^4),
$$
where $\hw$ is some function from $L^2_{\le 0}$. The additional adaptivity error $\errap$ is then
\begin{align*}
\lefteqn{\int_\T   \left| \hg_*^{\text{(ad)}}(u)-\frac{1}{ |\ell(e^{iu})|^2 +1}\right|^2
   \left( |\ell(e^{iu})|^2 +1  \right)  \mu(du)}
\\
&=
\int_\T   \left|1 + (\hw(u) + e^{-iu} -2) \alpha^2  + O(\alpha^4) -\frac{1}{ \alpha^2 |e^{iu} - 1 |^2 +1}\right|^2
\left(1 + O(\alpha^2) \right)  \mu(du)
\\
&\sim
\alpha^{4} \int_\T   \left|\hw(u) - e^{iu} \right|^2
\mu(du).
 \end{align*}
Thus, the function $\hw$ should be chosen as the solution to the classical prediction problem and we should have $\errap  \sim \alpha^4 \sigma^2_{\text{pred}}$.  This explains the similarity between \eqref{errap_as_al_to_0} and Kolmogorov's formula. Observe, finally, that by \eqref{errna_kin_discr},
$$
\errna  \sim    \al^2  \int_\T  |e^{iu}-1|^2 f(u)\, du
$$
as $\alpha\downarrow 0$. Thus, for small $\alpha$ the price for not knowing the future is small compared to the error of the non-adaptive approximation.
\bigskip

\subsubsection{Continuous time}

The approach and the result is very much the same as for stationary sequences except for
some integrability issues. We only replace
$\T$ with $\R$ and redefine the spaces $L^2_{\le 0}$ and $L^2_{>0}$ in $\L_2(\R,\Lambda)$ as the spaces of
Fourier transforms of functions supported by $\R_-$ and $\R_+$, respectively.

Again we use the class of {\it outer} functions, this time with respect to the lower half-plane,
cf.\ \cite[p.36]{DMK} and use the following
characterization, cf.\ \cite[p.39]{DMK}: a function $\gamma \in \L_2(\R,\Lambda)$ is an outer function
for the lower half-plane iff
\be \label{outer_c}
     \overline{\rm span}\{ \gamma e^{i\tau u},\ \tau\le 0\ \big|  \L_2(\R,\Lambda) \} =  L^2_{\le 0}.
\ee

The Kolmogorov regularity condition now looks as follows: $\mu(du)=f(u)du$ and
\be \label{Kolm_c}
   \int_\R \frac{|\ln f(u)|}{1+u^2} \ du < \infty.
\ee
This condition ensures the existence of the factorization
\[ 
  f(u)=   \gamma_f(u) \overline{\gamma_f(u)}= |\gamma_f(u)|^2,
  \qquad u \in \R,
\] 
with $\gamma_f$ being an outer function, cf.\ \cite[p.38]{DMK}.

For the energy function $\ell(\cdot)$ we need a factorization
\[ 
  |\ell(iu)|^2+1=   \lambda_\ell(u) \overline{\lambda_\ell(u)}= |\lambda_\ell(u)|^2 ,
  \qquad u \in \R,
\] 
with $\lambda_\ell$ satisfying properties \eqref{lambar} and \eqref{lam_hg}.
It was shown in \eqref{fac_lam3} above how to construct such factorization for polynomials.

Now the construction of the optimal adaptive approximation and the calculation of the
approximation error are done exactly as in the discrete-time case but we repeat
the approach for completeness of exposition.


\begin{thm}
Let $Q_{>0}$ be the orthogonal projection of $\gamma_f / \overline{\lambda_\ell}$ onto $L^2_{>0}$ in the Hilbert space $L^2 (\R, \Lambda)$.
Then the optimal adaptive approximation is given by $X(t) = \int_\R \hg_*(u) e^{itu} \W(du)$ with
$$
\hg_* (u) = \frac 1 {|\lambda_\ell|^2} - \frac{Q_{>0}}{\lambda_\ell \gamma_f}.
$$
The error of the adaptive approximation is given by
$
\errap = \|Q_{>0}\|_2^2
$.
\end{thm}
\begin{proof}
We have to compute
\[
   \errap =  \min_{\hg\in\hG}  \int_\R  \left| \hg(u)-\frac{1}{ |\ell(iu)|^2 +1}\right|^2
   \left( |\ell(iu)|^2 +1  \right)  \mu(du),
\]
where
\[
  \hG:= \overline{\rm span}\{ e^{i\tau u},\ \tau\le 0\ \big|  \L_2(\R,\mu) \}.
\]
By using factorizations, we have
\begin{align}
  \errap
  &=
  \min_{\hg \in \hG}  \int_\R  \left| \lambda_\ell(u)\, \hg(u) -  \frac{1}{\overline{\lambda_\ell(u)}} \right|^2  \mu(du) \label{errap_cL}\\
  &=
   \min_{\hg \in \hG}  \int_\R
  \left| ( \lambda_\ell \gamma_f \hg)(u)-\frac{\gamma_f(u)}{ \overline{\lambda_\ell(u) }} \right|^2 \ du.  \notag
\end{align}

Consider arbitrary $\hg\in\hG$. Without loss of generality we may assume that
$\lambda_\ell\, \hg\in \L_2(\R,\mu)$ (otherwise the integral in \eqref{errap_cL} is infinite).
Then by \eqref{lam_hg} we have $\lambda_\ell \, \hg\in \hG$, which is equivalent to
\[
  \lambda_\ell \gamma_f \hg
  \in \overline{\rm span}\{ \gamma_f e^{i\tau u},\ \tau\le 0 \big| \L_2(\R,\Lambda) \}
  = L^2_{\le 0},
\]
where the latter equality holds by \eqref{outer_c} because $\gamma_f$ is an outer function.

On the other hand, since  $\gamma_f\in L^2(\R,\Lambda)$ and
$|\lambda_\ell|\ge 1$, we have
\[
   Q:=\frac{\gamma_f}{ \overline{\lambda_\ell }} \in \L_2(\R,\Lambda).
\]
Consider the unique orthogonal decomposition in $\L_2(\R,\Lambda)$
\[
   Q:=\frac{\gamma_f}{ \overline{\lambda_\ell }} = Q_{\le 0}+  Q_{>0}
\]
with $Q_{\le 0}\in L^2_{\le 0}$ and $Q_{> 0}\in L^2_{>0}$.
Due to the orthogonality of the spaces $L^2_{\le 0}$ and $L^2_{>0}$,
we clearly have
\[
   \errap \ge \|Q_{>0}\|_2^2.
\]
Furthermore, the equality
\be \label{errQ_cont}
   \errap = \|Q_{>0}\|_2^2
\ee
is attained whenever
\be \label{gstar_c}
  \hg=\hg_* := \frac{Q_{\le 0}}{ \lambda_\ell \gamma_f}
  = \frac{Q- Q_{>0}}{ \lambda_\ell \gamma_f}
  = \frac 1{|\lambda_\ell|^{2}} -\frac{Q_{>0}}{ \lambda_\ell \gamma_f}.
\ee
It remains to prove that $\hg_*\in \hG$.
Since $\gamma_f$ is an outer function, we have by \eqref{outer_c}
\[
  Q_{\le 0} \in L^2_{\le 0}
  =\overline{\rm span}\{ \gamma_f e^{i\tau u},\ \tau\le 0 \big|   \L_2(\R,\Lambda) \},
\]
which is equivalent to
\[
  \frac{Q_{\le 0}} {\gamma_f}
  \in \overline{\rm span}\{  e^{i\tau u},\ \tau\le 0 \big|   \L_2(\R,\mu) \} = \hG.
\]
Finally, we obtain from \eqref{lambar} the required inclusion
\[
  \hg_* =  \frac{Q_{\le 0}} {\lambda_\ell \gamma_f} \in \hG,
\]
thus completing the proof.
\end{proof}
\medskip

For continuous-time kinetic energy $\ell (z) = \al z$, by taking the Fourier integral representation
\be \label{gammaj_c}
  \gamma_f(u) := \int_{0}^\infty \hgamma(\tau) e^{-i \tau u} d\tau,
  \qquad u\in \R,
\ee
and multiplying \eqref{lambinv_kc} and \eqref{gammaj_c}, we obtain
\begin{eqnarray*}
   Q(u) &=& \int_{0}^\infty \hgamma(\tau_1) e^{-i \tau_1 u} d\tau_1
      \cdot \frac{1}{\al}  \int_{0}^\infty  e^{-\tau_2/\al} e^{i \tau_2 u} d\tau_2
  \\
   &=&   \frac{1}{\al} \int_{0}^\infty \int_{0}^\infty
          \hgamma(\tau_1)  e^{-\tau_2/\al}   e^{i (\tau_2-\tau_1) u}   d\tau_1 d\tau_2
  \\
   &=&   \frac{1}{\al} \int_{-\infty}^\infty
         \left( \int_{\max(0,-\tau)}^\infty
         \hgamma(\tau_1)  e^{-(\tau+\tau_1)/\al}  d\tau_1 \right)   e^{i \tau u} d\tau.
\end{eqnarray*}
Hence,
\begin{eqnarray} \nonumber
   Q_{>0} (u) &=&  \frac{1}{\al} \int_{0}^\infty
         \left( \int_{0}^\infty
         \hgamma(\tau_1)  e^{-(\tau+\tau_1)/\al}  d\tau_1 \right)   e^{i \tau u} d\tau
\\
    &=& \nonumber
  \frac{1}{\al} \int_{0}^\infty  \hgamma(\tau_1)  e^{-\tau_1/\al}  d\tau_1
  \cdot  \int_{0}^\infty  e^{-\tau/\al} e^{i \tau u}  d\tau
\\   \label{Q_kc2}
  &:=&   \KK  \int_{0}^\infty  e^{-\tau/\al} e^{i \tau u}  d\tau
  =  \frac{\al \KK}{1-i\al u}
\end{eqnarray}
with
\be \label{Kf_c}
    \KK := \frac{1}{\al} \int_{0}^\infty  \widehat\gamma(\tau)  e^{-\tau/\al}  d\tau .
\ee
By  \eqref{errQ_cont} and \eqref{Q_kc2}, it follows that
\be  \label{err_kc}
  \errap = \|Q_{>0}\|_2^2 =  \al^2 |\KK|^2 \int_{\R} \frac{du}{1+\al^2u^2}
  =  \pi\, \al \, |\KK|^2.
\ee
Furthermore, by using \eqref{gstar_c} and \eqref{Q_kc2}, we obtain the continuous-time
analogue of \eqref{gstar_kd},
\be \label{gstar_kc}
    \hg_*(u) =
    \frac 1 {|\lambda_\ell(u)|^{2}} \left( 1 -\frac{\al \KK} {\gamma_f(u)}\right)
    =\frac 1 {1+\al^2 u^2} \left( 1 -\frac{\al \KK} {\gamma_f(u)}\right),
    \quad u\in\R.
\ee

To derive an explicit formula for the outer function $\gamma_f$ and the constant $\KK$ in terms of the spectral density $f$, consider the function
$$
q(z) := \exp\left\{\frac 1 {\pi i} \int_\R \frac{uz + 1}{u - z} \, \frac{\ln \sqrt{f(u)}}{u^2+1} \, du\right\}, \quad \Im z>0.
$$
It is known \cite[p.\ 37]{DMK} that $q(z)$ belongs to the Hardy space $H^2$ on the upper half-plane and the boundary limits of the absolute value $|q|$ are Lebesgue-a.e.\ given by
$$
\lim_{v\downarrow 0} |q(u+vi)| = \sqrt{f(u)}, \quad u \in\R.
$$
Defining
$$
\gamma_f(u) = \lim_{v\downarrow 0} \overline{q(u+vi)}, \quad u\in\R,
$$
we evidently have $|\gamma_f(u)|^2 = f(u)$ for $u\in\R$. Further, $\gamma_f$ satisfies \eqref{outer_c} by \cite[p.\ 37, p.\ 39]{DMK}. The Fourier representation of $\overline{\gamma_f}$, see \eqref{gammaj_c},  continues to hold in the upper half-plane:
$$
q(z) = \int_{0}^\infty \overline{\hgamma(\tau)} e^{ i \tau z} d\tau, \quad \Im z > 0.
$$
It follows from \eqref{Kf_c} that
$$
\KK = \frac 1 \al\overline{q(i/\alpha)} = \frac 1\al \exp\left\{\frac 1 {2\pi} \int_\R \frac{\alpha+ ui}{1+ ui\alpha} \, \frac{\ln f(u)}{u^2+1} \, du\right\}.
$$
Recalling \eqref{err_kc},  we arrive at the following
\begin{thm}
In the continuous-time case with $\ell(z) = \alpha z$, the additional adaptivity error is given by
\[
\errap = \frac{\pi}{\alpha} \exp\left\{ \frac \alpha {\pi} \int_{\R} \frac{ \ln f(u)}{1+\alpha^2 u^2} du \right\}.
\]
\end{thm}

As in the discrete-time case, if the Kolmogorov condition \eqref{Kolm_c} is violated (or if the spectral measure $\mu$ does not possess a density at all), the perfect prediction of the future is possible and we therefore have $\errap = 0$.

\medskip

\section{Examples of adaptive least energy approximations}
Unless the opposite is stated explicitly, in the following examples we consider kinetic energy,
i.e.\ we let $\ell(z)=\al z$ for continuous time and $\ell(z)=\al(z-1)$ for discrete time. Here
$\al>0$ is a fixed scaling parameter.

\subsection{Discrete time}

\subsubsection*{Autoregressive sequence}

A sequence of complex random variables $(B(t))_{t\in\Z}$ is called autoregressive,
if it satisfies the equation $B(t)=\rho B(t-1) + \xi(t)$, where $|\rho|<1$ and
$(\xi(t))_{t\in\Z}$
is a sequence of centered non-correlated complex random variables with
$\sigma^2:=\E |\xi(t)|^2$ not depending on $t$.
In this case we have a representation
\[
   B(t)= \sum_{j=0}^\infty \rho^j \xi(t-j), \qquad t\in \Z.
\]
For uncorrelated sequence we have a spectral representation
\be\label{spec_white_noise}
\xi(t) = \int_\T e^{i t u} \W(du)
\ee
where $\W$ is a complex centered random measure with uncorrelated values
on $\T$ controlled by the normalized Lebesgue measure
$\mu(du):= \tfrac{\sigma^2 du}{2\pi}$. Therefore, we obtain
\[
   B(t)=     \int_\T \sum_{j=0}^\infty \rho^j e^{i(t-j)u} \W(du)
        =     \int_\T  \frac{1}{ 1- \rho\, e^{-i u}} \ e^{i t u} \W(du).
\]
We see that the spectral measure  for $B$ is
\be \label{mu_ar}
    \mu(du):= \frac {\sigma^2 du}{2\pi|1- \rho\, e^{-i u}|^2},
\ee
which can be also found in Example 4.4.2 of \cite{BD}.
By \eqref{mu_ar} and  \eqref{errna_kin_discr}, the error of non-adaptive approximation
equals to
\be \label{errna_ar}
  \errna= \frac{\sigma^2}{1-|\rho|^2} \left( 1- \frac{1}{\sqrt{1+4\al^2}}\
	\frac{\b^2-|\rho|^2}{|\b-\rho|^2}  \right)
\ee
with $\b=\b(\al)$ defined in \eqref{beta}  (see \cite{IKL} for detailed calculation).

On the other hand, the spectral density factorizes as
\[
  f(u) = \frac{\sigma^2}{2\pi} |1-\rho e^{-iu}|^{-2} = \gamma_f(u) \overline{\gamma_f(u)}
\]
with
\be \label{gam_ar}
  \gamma_f(u) := \frac{\sigma}{\sqrt{2\pi}} (1-\rho e^{-iu})^{-1}
  =  \frac{\sigma}{\sqrt{2\pi}} \sum_{j=0}^\infty \rho^j e^{-iju}.
\ee
Hence, by  \eqref{Kf}
\be \label{Kf_ar}
   \KK = \frac{\sigma}{\sqrt{2\pi}}  \sum_{j=0}^\infty \rho^j \b^{-j}
   =  \frac{\sigma}{\sqrt{2\pi}(1-\rho/\b)},
\ee
whereas by  \eqref{err_kd}
\[ 
  \errap=  \frac{2\pi}{\b^2\sqrt{1+4\al^2}} \,  \frac{\sigma^2}{2\pi(1-\rho/\b)^2}
  =   \frac{\sigma^2}{\sqrt{1+4\al^2}\, |\b-\rho|^2 }.
\] 
Using  \eqref{errna_ar} we conclude that
\begin{eqnarray*}
   \erra&=& \errna+\errap
\\
   &=& \frac{\sigma^2}{1-|\rho|^2} \left( 1- \frac{1}{\sqrt{1+4\al^2}}\
	     \frac{\b^2-|\rho|^2}{|\b-\rho|^2} \right)
   +  \frac{\sigma^2}{\sqrt{1+4\al^2} \, |\b-\rho|^2 }\ .
\end{eqnarray*}
From  \eqref{gstar_kd}, \eqref{gam_ar}, and \eqref{Kf_ar}, for the optimal prediction we have
\begin{eqnarray} \nonumber
  \hg_*(u) &=& |\lambda_\ell|^{-2} \left( 1 -  \frac{\sigma}{\sqrt{2\pi}(1-\rho/\b)}
   \cdot \frac{ e^{iu} \sqrt{2\pi}(1-\rho e^{-iu})} { \b \sigma }\right)
\\ \nonumber
   &=& |\lambda_\ell|^{-2} \left( 1 -  \frac{e^{iu}-\rho}{\b-\rho} \right)
   = |\lambda_\ell|^{-2}   \frac{\b - e^{iu}}{\b-\rho}
\\ \nonumber 
   &=&   \frac{\b}{\al^2 (\b-\rho)(\b  - e^{-iu})}
   =   \frac{1}{\al^2(\b-\rho)} \sum_{j=0}^\infty \b^{-j} e^{-iju}.
\end{eqnarray}

\subsubsection*{Uncorrelated sequence}
Consider an uncorrelated sequence as a special case of the autoregressive
one with $\rho=0$. The best adaptive approximation is given by
\[
   \hg_*(u) =  \frac{1}{\al^2} \sum_{j=0}^\infty \b^{-j-1} e^{-iju}
\]
and the approximation errors are
\begin{align*}
  \errap&= \frac{\sigma^2}{\b^2\sqrt{1+4\al^2}},\\
  \erra&= \sigma^2 \left( 1- \frac{1}{\sqrt{1+4\al^2}}  \right)
   +  \frac{\sigma^2}{\b^2\sqrt{1+4\al^2}}
 =
   \sigma^2 \left( 1- \frac{1}{\al^2\b}  \right).
\end{align*}
Here we used again the identity \eqref{id_bmb} in the last step.

\subsubsection*{Simplest moving average sequence}

We call a sequence of complex random variables  $(B(t))_{t\in\Z}$
a simplest moving average sequence if it admits a representation
$B(t)= \xi(t)+\rho\, \xi(t-1)$, where $(\xi(t))_{t\in\Z}$ is a sequence of centered
non-correlated complex random variables with  $\sigma^2:=\E |\xi(t)|^2$
not depending on $t$.

Using \eqref{spec_white_noise}, we obtain
\[
   B(t)=  \int_\T  \left(1+\rho e^{-iu} \right) e^{itu} \W(du),
   \qquad t\in \Z.
\]
We conclude that the spectral measure for $B$ is
\be\label{mu_ma1}
   \mu(du):= \frac {\sigma^2 |1+ \rho\, e^{-iu}|^2 du}{2\pi};
\ee
see  Example 4.4.1 in~\cite{BD}.
By  \eqref{mu_ma1} and \eqref{errna_kin_discr},
the error of non-adaptive approximation equals to
\be \label{errna_ma}
  \errna= \sigma^2 \left( 1+|\rho|^2 - \frac{1}{\sqrt{1+4\al^2}}
	\left(  1+|\rho|^2 + \frac{\rho+\overline{\rho}}{\b}  \right)\right)
\ee
with $\b=\b(\al)$ defined in \eqref{beta} (see \cite{IKL} for detailed
calculation).

The form of factorization of spectral density depends on $|\rho|$.
If  $|\rho|<1$, then we have the factorization
\[
  f(u) = \frac{\sigma^2}{2\pi}\,  \big|1+\rho e^{-iu}\big|^{2}
  = \gamma_f(u) \, \overline{\gamma_f(u)}
\]
with
\[
  \gamma_f(u) := \frac{\sigma}{\sqrt{2\pi}} (1+\rho e^{-iu}).
\]
Hence, by \eqref{Kf}
\[
   \KK = \frac{\sigma}{\sqrt{2\pi}}  \left( 1+ \frac{\rho}{\b} \right),
\]
whereas by  \eqref{err_kd}
\[
  \errap =  \frac{2\pi}{\b^2\sqrt{1+4\al^2}} \,
	\frac{\sigma^2 \left| 1+ \frac{\rho}{\b} \right|^2}{2\pi}
  =   \frac{\sigma^2 \left|1+ \frac{\rho}{\b}\right|^2 }{\b^2 \sqrt{1+4\al^2} }.
\]
Using  \eqref{errna_ma} we arrive at
\begin{eqnarray*}
  && \erra = \errna+\errap
\\
   &=&    \sigma^2 \left( 1+|\rho|^2 - \frac{1}{\sqrt{1+4\al^2}}
	       \left(  1+|\rho|^2 + \frac{\rho+\overline{\rho}}{\b}  \right)\right)
   +   \frac{\sigma^2 \left|1+ \frac{\rho}{\b}\right|^2 }{\b^2 \sqrt{1+4\al^2} }
\\
    &=&    \sigma^2 \left(  1+|\rho|^2- \frac{1}{\b\al^2} - \frac{\rho+\overline{\rho}}{\b^2\al^2}
      -\frac{|\rho|^2(2\al^2+1)}{\b^2 \al^4}  \right).
\end{eqnarray*}
In our setting,
\begin{eqnarray*}
   \frac{\gamma_f}{\overline{\lambda_\ell}} (u)
   &=& \frac{\sigma}{\sqrt{2\pi}} (1+\rho e^{-iu}) \frac{\sqrt{\b}}{\al} \ \frac{1}{e^{iu}-\b}
\\
   &=& \frac{-\sigma}{\sqrt{2\pi}\al\sqrt{\b}} (1+\rho e^{-iu}) \frac{1}{1-e^{iu}/\b}
\\
   &=& \frac{-\sigma}{\sqrt{2\pi}\al\sqrt{\b}} (1+\rho e^{-iu}) \sum_{j=0}^\infty \b^{-j}e^{i j u}.
\end{eqnarray*}
It follows that
\[
  Q_{\le 0} (u) = \frac{-\sigma}{\sqrt{2\pi}\al\sqrt{\b}}
  \left( \rho e^{-i u}+ 1+\frac{\rho}{\b} \right).
\]
From  \eqref{gstar_kd} we find the optimal prediction
\[
  \hg_* (u)= \frac{Q_{\le 0}}{ \lambda_\ell \gamma_f}
  =    \frac{-(\rho e^{-i u}+ 1+\frac{\rho}{\b})}{\al^2 (1+\rho e^{-iu}) (e^{-iu}-\b) }\, .
\]
If $|\rho|<1$, $\rho\not=-1/\b$ we may expand this expression as
\begin{eqnarray*}
  && \hg_*(u) = \frac{1}{\al^2(\rho\b+1)}\left[ \frac{\rho^2}{\b(1+\rho e^{-iu})} -
\frac{1+\tfrac{\rho}{\b}+\rho\b}{e^{-iu}-\b}
  \right]
\\
  &=& \frac{1}{\al^2\b(\rho\b+1)}\left[ \frac{\rho^2}{1+\rho e^{-iu}} +
\frac{1+\tfrac{\rho}{\b}+\rho\b}{1-e^{-iu}/\b}
  \right]
\\
&=& \frac{1}{\al^2\b(\rho\b+1)} \sum_{j=0}^\infty
     \left[ (-1)^j\rho^{j+2}+\b^{-j}+\rho\b^{-j-1}+\rho \b^{-j+1}
     \right] e^{-i j u}.
\end{eqnarray*}

Notice that when letting $\rho=0$ we are back to the results for uncorrelated variables.

In the case $|\rho|>1$, we have
\begin{eqnarray*}
  \gamma_f(u) &:=& \frac{\sigma}{\sqrt{2\pi}} (\overline{\rho}+ e^{-iu}),
\\
   \KK &=& \frac{\sigma}{\sqrt{2\pi}}  \left( \overline{\rho}+ \frac{1}{\beta}\right),
\\
  \errap &=&   \frac{\sigma^2 \left|\rho+ \frac{1}{\beta}\right|^2 }{\b^2 \sqrt{1+4\al^2}}\, ,
\\
  \erra &=&   \sigma^2 \left(  1+|\rho|^2- \frac{2\al^2+1}{\b^2 \al^4}  - \frac{\rho+\overline{\rho}}{\b^2\al^2}
	      -\frac{|\rho|^2}{\b\al^2} \right).
\end{eqnarray*}
Furthermore, we have
\begin{eqnarray*}
    \hg_* (u) &=&  \frac{\overline{\rho}+\frac{1}{\b}+ e^{-iu}}{\al^2 (\b-e^{-iu}) (\overline{\rho} + e^{-iu})}
		\\
		&=& \left(1+ \frac{1}{\b(\overline{\rho}+\b)}\right) \frac{1}{\al^2 (\b-e^{-iu})}
		   +   \frac{1}{\b(\overline{\rho}+\b)}\  \frac{1}{\al^2 (\overline{\rho} + e^{-iu})}
\end{eqnarray*}
(the latter formula being valid if $\rho\not=-\beta$). Finally, we obtain an expansion
\[
    \hg_* (u) =  \left(1+ \frac{1}{\b(\overline{\rho}+\b)}\right)\, \frac{1}{\al^2\b}
		              \sum_{j=0}^\infty \frac{e^{-iju}}{\b^j}
									+  \frac{1}{(\overline{\rho}+\b)\al^2\b\overline{\rho}}
		                \sum_{j=0}^\infty \frac{e^{-iju}}{(-\overline{\rho})^j}.	
\]

\subsection{Continuous time}

\subsubsection*{Ornstein--Uhlenbeck process}

The Ornstein--Uhlenbeck process is a centered Gaussian stationary process with covariance $K_B(t)=e^{-|t|/2}$ and
the spectral measure
\be \label{mu_OU}
    \mu(du):= \frac {2 du}{\pi(4u^2+1)}.
\ee

By \eqref{mu_OU} and \eqref{errna_kin}, the error
of non-adaptive approximation is easy to calculate as
$\errna = \frac{\al}{2+\al}$.

The spectral density factorizes as
\[
  f(u) = \frac{2}{\pi(4u^2+1)} = \gamma_f(u) \overline{\gamma_f(u)}
\]
with
\[
  \gamma_f(u) := \sqrt{\frac{2}{\pi}} \frac{1}{1+2iu}
  =  \int_{0}^\infty    \frac{1}{\sqrt{2\pi}} \  e^{-\tau/2} e^{-i \tau u} d\tau.
\]
Hence, by  \eqref{Kf_c}
\[
   \KK =  \frac{1}{\al} \cdot \frac{1}{\sqrt{2\pi}}  \int_{0}^\infty   e^{-\tau/2} e^{-\tau/\al}  d\tau
   = \frac{\sqrt{2}}{\sqrt{\pi}(2+\al)} ,
\]
whereas by  \eqref{err_kc}
\[
  \errap=     \pi\, \al \, \frac{2}{\pi(2+\al)^2} =    \frac{2\al}{(2+\al)^2} ,
\]
and  we have
\[
   \erra= \errna+\errap = \frac{\al}{2+\al} + \frac{2\al}{(2+\al)^2} =  \frac{4\al+\al^2}{(2+\al)^2} \ .
\]
For the optimal adaptive approximation, we easily obtain from \eqref{gstar_kc}
\[
   \hg(u)= \frac{2}{2+\al}\, \frac{1}{1+i \al u}\,,
\]
hence, the optimal weight is
\[
  g(\tau)= \frac{2}{(2+\al)\al}\, e^{\tau/\al}\, \ed{\tau\le 0}.
\]
Summarizing, we arrive at the following
\begin{thm}
Let $\ell(z)= \alpha z$. The optimal adaptive approximation of the Ornstein--Uhlenbeck process with covariance function $K_B(t)=e^{-|t|/2}$ is given by
\[
    X(t)= \frac{2}{(2+\al)\al}\,  \int_{0}^\infty B(t-s)\, e^{-s/\al}\, ds,
\]
and the corresponding error is $\erra= (4\al+\al^2)/(2+\al)^2$.
\end{thm}
The same results may be formally obtained by discretization of Orn\-stein--Uhlenbeck process which leads
to autoregressive sequence with parameters $\rho_\delta=e^{-\delta/2}, \al_\delta=\tfrac{\al}{\delta},
\sigma_\delta^2=1-\rho_\delta^2$ and letting $\delta\to 0$.
\medskip



\section*{Acknowledgments}

This research was supported by DFG--SPbSU grant 6.65.37.2017. The work of the second named
author was also supported by RFBR grant 16-01-00258.


\end{document}